\documentclass[12pt]{article}
\usepackage[UKenglish]{babel}
\usepackage[T1]{fontenc}
\usepackage{lmodern,amsmath,amsthm,amsfonts,amssymb,graphicx,float,microtype,thmtools,underscore,mathtools,thm-restate}
\usepackage[shortlabels]{enumitem}
\setlist[itemize]{topsep=0ex,itemsep=0ex,parsep=0ex}
\setlist[enumerate]{topsep=0ex,itemsep=0ex,parsep=0ex}
\usepackage[usenames,dvipsnames,svgnames,table]{xcolor}
\usepackage{todonotes}
\usepackage[unicode=true]{hyperref}

\hypersetup{ 
colorlinks,
linkcolor={blue!60!black},
citecolor={black},
breaklinks=true,
urlcolor={blue!60!black},
pdftitle={Coarse Balanced Separators and Tree-Decompositions}}
\usepackage{xurl}
\usepackage[capitalise, compress, nameinlink, noabbrev]{cleveref}
\crefname{lem}{Lemma}{Lemmas}
\crefname{thm}{Theorem}{Theorems}
\crefname{ques}{Question}{Theorems}
\crefname{cor}{Corollary}{Corollaries}
\crefname{enumi}{Item}{Items}
\crefformat{equation}{(#2#1#3)}
\Crefformat{equation}{Equation #2(#1)#3}
\crefformat{enumi}{#2#1#3}
\Crefformat{enumi}{Item (#2#1#3)}

\newcommand{\defn}[1]{\textcolor{Maroon}{\emph{#1}}}
\usepackage[longnamesfirst,numbers,sort&compress]{natbib}
\makeatletter
\def\NAT@spacechar{~}
\makeatother
\usepackage[margin=28mm]{geometry}
\renewcommand{\baselinestretch}{1.1}
\allowdisplaybreaks

\renewcommand{\epsilon}{\varepsilon}

\renewcommand{\geq}{\geqslant}
\renewcommand{\leq}{\leqslant}

\DeclareMathOperator{\dist}{dist}
\DeclareMathOperator{\tw}{tw}

\DeclareMathOperator{\sep}{sep}

\newcommand{\RR}{\mathbb{R}}

\newcommand{\GG}{\mathcal{G}}

\newcommand{\NN}{\mathbb{N}}

\renewcommand{\thefootnote}{\fnsymbol{footnote}}
\theoremstyle{plain}
\newtheorem{thm}{Theorem}
\newtheorem{lem}[thm]{Lemma}

\crefname{obs}{Observation}{Observations}
\newtheorem*{lem*}{Lemma}
\theoremstyle{definition}
\newtheorem{conj}[thm]{Conjecture}
\newtheorem*{conj*}{Conjecture}

\presetkeys%
{todonotes}%
{inline}{}
\date{}

\begin{document}

\title{\bf\fontsize{18pt}{18pt}\selectfont Coarse Balanced Separators and Tree-Decompositions}

\author{Maria Chudnovsky\,\footnotemark[1] \quad Robert Hickingbotham\,\footnotemark[2]}

\footnotetext[1]{Princeton University, Princeton, NJ, USA. Supported by NSF Grant DMS-2348219 and by AFOSR grant FA9550-22-1-0083.}

\footnotetext[2]{CNRS, ENS de Lyon, Université Claude Bernard Lyon 1, LIP, UMR 5668, Lyon, France.}

\maketitle
\begin{abstract}
  A classical result of Robertson and Seymour (1986) states that the treewidth of a graph is linearly tied to its separation number: the smallest integer $k$ such that, for every weighting of the vertices, the graph admits a balanced separator of size at most $k$. Motivated by recent progress on coarse treewidth, Abrishami, Czy{\.{z}}ewska, Kluk, Pilipczuk,  Pilipczuk, and Rz{\k{a}\.{z}}ewski (2025) conjectured a coarse analogue to this result: every graph that has a balanced separator consisting of a bounded number of balls of bounded radius is quasi-isometric to a graph with bounded treewidth. In this paper, we confirm their conjecture for $K_{t,t}$-induced-subgraph-free graphs when the separator consists of a bounded number of balls of radius $1$. In doing so, we bridge two important conjectures concerning the structure of graphs that exclude a planar graph as an induced minor.
\end{abstract}

\renewcommand{\thefootnote}{\arabic{footnote}}

\section{Introduction}
Coarse graph theory is an emerging field that explores the global structure of graphs through the lens of Gromov’s coarse geometry.\footnote{See \cref{SectionPrelim} for undefined terms.} Initiated by \citet{georgakopoulos2023graph}, this area seeks to understand the global structure of graphs by viewing them from afar. In this paper, we continue this line of research by developing coarse analogues of two fundamental concepts: balanced separators and tree-decompositions.

We begin with the definition of treewidth. A \defn{$T$-decomposition} $\mathcal{T}={(T,\beta)}$ of a graph $G$ is a pair where $T$ is a tree and ${\beta \colon V(T) \to 2^{V(G)}}$ is a function such that:
 \begin{itemize}
     \item for every edge ${uv \in E(G)}$, there exists a node ${x \in V(T)}$ with ${u,v \in \beta(x)}$; and 
     \item for every vertex ${v \in V(G)}$, the set $\{ x \in V(T) \colon v \in \beta(x) \}$ induces a non-empty connected subtree of~$T$. 
 \end{itemize}
We call $\beta(t)$ a \defn{bag} of the $T$-decomposition. The \defn{width} of ${(T,\beta)}$ is ${\max\{ \lvert \beta(x)\rvert \colon x \in V(T) \}-1}$. The \defn{treewidth}~$\tw(G)$ of $G$ is the minimum width of a $T$-decomposition of~$G$ for any tree $T$. Treewidth is an important parameter in algorithmic and structural graph theory that measures how similar a graph is to a tree; see the survey \cite{HW2017tied}.

Recently, the second author~\cite{Hickingbotham2025treewidth} and \citet{NSS2025treewidth} independently established the right notion of coarse treewidth (with optimal bounds obtained by \cite{NSS2025treewidth}). Let $G$ be a graph, and $k,r\in \NN$.  We say that a set $S\subseteq V(G)$ is \defn{$(k,r)$-centred} if $S\subseteq N_G^r[\widehat{S}]$ for some set $\widehat{S}\subseteq V(G)$ with $|S|\leq k$. Note that distances are measured in the underlying graph $G$, not by the subgraph induced by $S$. A tree-decomposition of $G$ is \defn{$(k,r)$-centred} if each bag of the tree-decomposition is $(k,r)$-centred. 
  
\begin{thm}[\cite{NSS2025treewidth,Hickingbotham2025treewidth}]\label{CoarseTreewidthCharacterisation}
    For all $k,r\in \NN$, there exists $q\in \NN$ such that every graph that has a $(k,r)$-centred tree-decomposition is $q$-quasi-isometric to a graph with treewidth at most $k$.
\end{thm}

Quasi-isometry is a fundamental notion from metric geometry which captures when two metric spaces share the same global structure. It is straightforward to show the converse to \cref{CoarseTreewidthCharacterisation} hold: that every graph quasi-isometric to a graph with treewidth at most $k$ has a $(k+1,r)$-centred tree-decomposition (for an appropriate choice of $r$). Thus, \cref{CoarseTreewidthCharacterisation} exactly characterises ``coarse treewidth.'' Given that we now have a robust measure of coarse treewidth, it raises the question as to which properties of treewidth can be lifted to this setting.

A classical result of \citet{robertson1986algorithmic} ties treewidth to the existence of small balanced separators. For a graph $G$ and vertex-weight function $\mu\colon V(G)\to \RR_{\geq 0}$, we say that a set $S$ is a \defn{balanced separator} for $\mu$ if, for every connected component $C$ of $G-S$, the total weight of vertices within $C$ is at most half the total weight of vertices within $G$. If $\mu$ is an indicator function for a set $X\subseteq V(G)$, then we say that $S$ is a \defn{balanced separator} for $X$. The \defn{separation number} $\sep(G)$ of $G$ is define to be the minimum $k\in \NN$ such that, for every $\mu\colon V(G)\to \RR_{\geq 0}$, there is a balanced separator $S$ for $\mu$ with size at most $k$. A standard argument shows that $\sep(G)\leq \tw(G)+1$ for every graph $G$. Conversely, \citet{robertson1986algorithmic} proved that the treewidth of a graph is bounded from above by a linear function of its separation number.

\begin{thm}[\cite{robertson1986algorithmic}]\label{BalancedSep}
    For every graph $G$, $\tw(G)\leq 4\sep(G).$
\end{thm}

Inspired by this, \citet{ACKPPR2025Balanced} conjectured that this characterisation of treewidth has a coarse analogue. We say that a graph $G$ admits \defn{$(k,r)$-balanced separators} if, for every vertex-weight function $\mu\colon V(G)\to \RR_{\geq 0}$, there is a $(k,r)$-centred set which is a balanced separator for $\mu$. 

\begin{conj}[\cite{ACKPPR2025Balanced}]\label{MainConjecture}
    For all $k,r\in \NN$, there exist $k',r'\in \NN$ such that every graph that admits $(k,r)$-balanced separators has a $(k',r')$-centred tree decomposition.
\end{conj}

The same argument that shows $\sep(G)\leq \tw(G)+1$ also implies that every graph that has a $(k,r)$-centred tree decomposition admits $(k,r)$-balanced separators. Therefore, this conjecture, if true, would provide a balanced separator characterisation for coarse treewidth.

In this paper, we verify \cref{MainConjecture} for hereditary $K_{t,t}$‑induced-subgraph-free graph classes in the $r=1$ case.

\begin{thm}\label{MainTheorem}
    For all $k,t\in \NN$, there exists $k'\in \NN$ such that the following holds: Let $G$ be a $K_{t,t}$-free graph such that every induced subgraph of $G$ admits $(k,1)$-balanced separators. Then $G$ has a $(k',2)$-centred tree-decomposition.
\end{thm}
Together with \cref{CoarseTreewidthCharacterisation}, this gives a natural sufficient condition for a graph to be quasi-isometric to a graph with bounded treewidth.

Note that our proof for \cref{MainTheorem} gives the following stronger property: the radius-$2$ neighbourhoods can be taken with respect to the subgraph induced by the corresponding bag of the tree-decomposition. 

To further motivate \cref{MainTheorem}, we place it in the broader context of recent developments concerning induced minors. There has been growing interest in understanding which hereditary graph classes admit $(k,1)$-balanced separators \cite{CGHLS2024Evenhole,CHLS2024ThreePath,CCLMS2025Walls,Gartland2024MWIS,GartlandLokshtanov2020,Bacso2019Subexponential}. This line of work is driven by the goal of designing quasi-polynomial time algorithms for maximum independent set, since the existence of such separators are useful for constructing such algorithms.

Note that if a graph has bounded maximum degree and admit $(k,1)$-balanced separators, then it has bounded separation number, and hence has bounded treewidth. Consequently, a necessary condition for a hereditary graph class to admit $(k,1)$-balanced is that it needs to exclude a large grid as an induced minor. Gartland and Lokastov~\cite{Gartland23} conjecture that this is, in fact, the only obstruction:

\begin{conj}[\cite{Gartland23}]\label{InducedMinorSep}
    For every planar graph $H$, there exists $k\in \NN$ such that every $H$-induced-minor-free graph $G$ admits $(k,1)$-balanced separator.
\end{conj}

This conjecture has been confirmed in several cases: $P_t$-free graphs \cite{Bacso2019Subexponential}; even-hole-free graphs \cite{CGHLS2024Evenhole}; three-path-configuration-free graphs \cite{CHLS2024ThreePath}; and graphs that exclude the line graph of subdivisions of a wall and a subdivided claw as induced subgraphs \cite{CCLMS2025Walls}.

Related to \cref{InducedMinorSep} is the \emph{Coarse Grid Minor Conjecture} due to \citet{georgakopoulos2023graph}.\footnote{Note that \cite{georgakopoulos2023graph} states their conjecture in terms of forbidden fat-minors, which is more general than forbidden induced minors.} This conjecture is one of the most important open problems in coarse graph theory. 

\begin{conj}[\cite{georgakopoulos2023graph}]\label{InducedMinorTreewidth}
    For every planar graph $H$, there exist $k,q\in \NN$ such that every $H$-induced-minor-free is $q$-quasi-isometric to a graph with treewidth at most $k$.
\end{conj}

Since excluding a graph as an induced minor is closed under taking induced subgraphs, a consequence of \cref{MainTheorem} is that \cref{InducedMinorSep} implies \cref{InducedMinorTreewidth} for $K_{t,t}$-free graphs. Thus, our result provides a bridge between these two important conjectures concerning the structure of graphs that exclude a planar graph as an induced minor.

\section{Preliminaries}\label{SectionPrelim}
All graphs in this paper are simple and finite. Undefined terms and standard definitions can be found in \citet{diestel2017graphtheory}.

Let $G$ and $H$ be graphs. The $2$-subdivision of $H$, denoted $H^{(2)}$, is the graph obtained from $H$ by replacing each edge with a path of length $3$. We say that $H$ is an~\defn{induced subgraph} of $G$ if $H$ can be obtained from $G$ by deleting vertices. We say that $G$ is \defn{$H$-free} if $G$ does not contain an induced subgraph isomorphic to $H$. A class of graphs $\GG$ is \defn{hereditary} if it is closed under taking induced subgraphs. 

For a set $S \subseteq V(G)$, let $G[S]$ denote the graph obtained by removing from $G$ all the vertices that are not in $S$. We write $G-S$ as short-hand for $G[V(G)\setminus S]$. We say that $H$ is an \defn{induced minor} of $G$ if $H$ is isomorphic to a~graph that can be obtained from an induced subgraph of $G$ by contracting edges. We say that $G$ is \defn{$H$-induced-minor-free} if $H$ is not an induced-minor of $G$. 

Two sets $X,Y\subseteq V(G)$ are \defn{anti-complete} if they are disjoint and there is no edge in $G$ with one end-point in $X$ and the other in $Y$. An \defn{independent set} $S\subseteq V(G)$ is a set of vertices such that no two vertices in $S$ are adjacent. The \defn{independence number} $\alpha(G)$ of $G$ is the size of the largest independent set in $G$. For a set $X\subseteq V(G)$, we may abuse notation and write $\alpha(X)$ for $\alpha(G[X])$.

The \defn{distance} $\dist_G(u,v)$ between $u$ and $v$ in $G$ is the length of the shortest path connecting them, or infinite otherwise. For $r\in \NN$ and set $\hat{S}\subseteq V(G)$, let $N^r[\hat{S}]$ denote the set of vertices in $G$ at distance at most $r$ from $S$. We may drop the subscript $G$ when the graph is clear from context, and we may write $N[\hat{S}]$ instead of $N^1[\hat{S}]$. 

For $q \in \NN$, a \defn{$q$-quasi-isometry} of $G$ into a graph $H$ is a map $\phi \colon V(G) \to V(H)$ such that, for every $u,v\in V(G)$,
\begin{equation*}
    q^{-1} \cdot \dist_G(u,v) - q \leq \dist_{H}(\phi(u), \phi(v))\leq q \cdot \dist_G(u,v) + q,
\end{equation*}
and, for every $x \in V(H)$, there exists a vertex $v \in V(G)$ such that $\dist_{H}(x,\phi(v)) \leq q$. If such a map exists, then we say that $G$ is \defn{$q$-quasi-isometric} to $H$. 

\section{Proof}

We make no attempt to optimise the constants in our bounds. We need the following lemma from \citet{CCLMS2025Walls}.

\begin{lem}[\cite{CCLMS2025Walls}]\label{SmallAlphaNeighbours}
    Let $C,\gamma,t\in \NN$ be such that $C,\gamma \geq 2$, and let $G$ be a $\{K_{\gamma}^{(2)},K_{t,t}\}$-free graph. Let $Y\subseteq V(G)$. Define
    $$Z=\{v\in V(G)\colon \alpha(N[v]\cap Y)\geq \frac{\alpha(Y)}{C}\}.$$
    Then $\min\{\alpha(Y),\alpha(Z)\}\leq (512C)^{\gamma^{2t}}$.
\end{lem}

By assumption, the graph $G$ in \cref{MainTheorem} is $K_{t,t}$-free. The next lemma shows that $G$ is also $K_{2k+2}^{(2)}$-free, which will allow us to apply \cref{SmallAlphaNeighbours}.

\begin{lem}\label{ObsK2k}
    For every $k\in \NN$, the graph $K_{2k+2}^{(2)}$ does not admit balanced $(k,1)$-balanced separators.
\end{lem}

\begin{proof}
Let $X$ be the set of high-degree vertices in $K_{2k+2}^{(2)}$. Then $|X|=2k+2$. Suppose, for contradiction, that there is a set $\hat{S}$ of at most $k$ vertices in $K_{2k+2}^{(2)}$ whose closed neighbourhood $S=N[\hat{S}]$ is a balanced separator for $X$. Let $X'=X-(X\cap S)$. Since no pair of vertices in $X$ lie in a common ball of radius $1$ in $K_{2k+2}^{(2)}$ (as they are pairwise distance-$3$ apart), it follows that $|X'|\geq 2k+2-k = k+2$. Choose distinct $x,y\in X'$ and let $(x,w_1,w_2,y)$ be the path of length $3$ joining them in $K_{2k+2}^{(2)}$. If either $w_1$ or $w_2$ were in $S$, then $x$ or $y$ would also be in $S$, contradicting their membership in $X'$. Thus, $x$ and $y$ are in the same component $C$ of $K_{2k+2}^{(2)}-S$. Since this argument holds for every pair of vertices in $X'$, it follows that $X'\subseteq V(C)$ and so $C$ contains at least $k+2 > |X|/2$ vertices from $X$, contradicting the assumption that $S$ is a balanced separator for $X$.
\end{proof}

The following is our main technical result, which immediately implies \cref{MainTheorem}. 

\begin{lem}
    Let $k,t\in \NN$ and define $d=(512\cdot 20k)^{{(2k+2)}^{2t}}.$ Let $G$ be a $K_{t,t}$-free graph such that every induced subgraph of $G$ admits $(k,1)$-balanced separators. Then, for every $X\subseteq V(G)$ with $\alpha(X)\leq 10dk$, there exists a $(20dk,2)$-centred tree-decomposition of $G$ which has a bag that contains $N[X]$.
\end{lem}

\begin{proof}
    We proceed by induction on $\alpha(G)$. Since every induced subgraph of $G$ admits balanced $(k,1)$-separators, \cref{ObsK2k} implies that $G$ is $K_{2k+2}^{(2)}$-free. If $\alpha(G) \leq 20dk$, then taking the tree-decomposition of $G$ in which $V(G)$ is in a single bag satisfies the statement. Hence, we may assume that $\alpha(G) > 20dk$.

    If $\alpha(X) < 10dk$, then enlarge $X$ by adding vertices into it until $\alpha(X)=10dk$. Define 
    $$Z_G=\{v\in V(G)\colon \alpha(N[v])\geq \frac{1}{20k}\alpha(V(G))\}$$ 
    and 
    $$Z_X=\{v\in V(G)\colon \alpha(N[v]\cap X)\geq \frac{1}{20k}\alpha(X)\}.$$ 
    Let $Z=Z_G\cup Z_X$. Since both $\alpha(X)$ and $\alpha(G)$ are greater than $d$, we may apply \cref{SmallAlphaNeighbours} to conclude that $\alpha(Z)\leq 2d$. 
    
    Let $G'=G-Z$ and $X'=V(G')\cap X$. Then, for every vertex $v\in V(G')$, we have
    $$\alpha(N_{G'}[v]) <\frac{1}{20k}\alpha(G) \text{\quad and \quad} \alpha(N_{G'}[v]\cap X') < \frac{1}{20k}\alpha(X),$$
    otherwise $v$ would be in $Z$.
    Let $I_{G}$ be a maximum independent set of $G'$ and let $I_X$ be a maximum independent set of $X'$. Since $\alpha(Z)\leq 2d$, it follows that $|I_G|\geq \alpha(G)-2d$ and $|I_X|\geq \alpha(X)-2d$. Since $G'$ is an induced subgraph of $G$, there exist sets $\hat{S}_G,\hat{S}_X\subseteq V(G')$ with $\max\{|\hat{S}_G|,|\hat{S}_G|\}\leq k$ such that $S_G=N_{G'}[\hat{S}_G]$ is a balanced separator for $I_G$ and $S_X=N_{G'}[\hat{S}_X]$ is a balanced separator for $I_X$. Define $\hat{S}=\hat{S}_G \cup \hat{S}_X$ and $S = S_G\cup S_X$. Since every vertex in $G'$ is adjacent to fewer than $\frac{1}{20k}(|I_G|+2d)$ vertices from $I_G$ (and similarly for $I_X$), it follows that 
    $$|I_G\cap S|\leq \frac{1}{20k}(|I_G|+2d) \cdot 2k= \frac{1}{10}(|I_G|+2d)$$
    and
    $$|I_X\cap S|\leq \frac{1}{20k}(|I_X|+2d) \cdot 2k= \frac{1}{10}(|I_X|+2d).$$
    
    Let $C$ be a component of $G'-S$ and define $\widehat{C}=G'-(S\cup V(C))$. By the balancing properties of the separators,
    $$|V(C)\cap I_G|\leq \frac{1}{2} |I_G| \text{\quad and \quad} |V(C)\cap I_X|\leq \frac{1}{2} |I_X|.$$ 
    Therefore,
    $$|V(\widehat{C})\cap I_G|\geq |I_G|-\frac{1}{2} |I_G|-\frac{1}{10}(|I_G|+2d)=\frac{1}{5}(2|I_G|-d)$$
    
    and
    $$|V(\widehat{C})\cap I_X|\geq |I_X|-\frac{1}{2} |I_X|-\frac{1}{10}(|I_X|+2d)=\frac{1}{5}(2|I_X|-d).$$
    
    Since $V(C)$ and $V(\widehat{C})$ are anti-complete, we deduce that
    $$\alpha(C)\leq \alpha(G)-\alpha(\widehat{C})\leq \frac{1}{5}(3\alpha(G)+d) \text{\quad 
    and \quad}
    \alpha(C\cap X)\leq \alpha(X)-\alpha(\widehat{C}\cap X)\leq \frac{1}{5}(3\alpha(X)+d).$$
    
    Set $C'=G[V(C)\cup S \cup Z]$ and $X_C=(V(C\cup S)\cap X)\cup \hat{S} \cup Z$. Then
    $$\alpha(C')\leq \alpha(C)+\alpha(S)+\alpha(Z)\leq \frac{1}{5}(3\alpha(G)+d)+\frac{1}{10}\alpha(G)+2d<\alpha(G)$$
    
    and similarly,
    $$\alpha(X_C)\leq \alpha(C\cap X)+\alpha(S\cap X)+|\hat{S}|+\alpha(Z)\leq \frac{1}{5}(3\alpha(X)+d)+\frac{1}{10}\alpha(X)+2k+2d\leq 10dk.$$
    
    By the inductive hypothesis, the graph $C'$ admits a $(20dk,2)$-centred tree-decomposition $\mathcal{T}_{C'}$ such that $N_{C'}[X_C]$ is contained in a single bag. For each component $C$ of $G'-S$, add a leaf bag to the tree-decomposition $\mathcal{T}_{C'}$ of $C'$ containing $N_G[X]\cup S\cup Z$ adjacent to the bag containing $N_{C'}[X_C]$. Since $N_G[X]\cap V(C')\subseteq N_{C'}[X_C]$, this defines a tree-decomposition of $G[V(C')\cup N_G[X]]$. Since $\alpha(X)\leq 10dk$, the set $N_G[X]\cup S\cup Z$ is $(10dk+2d+2k,2)$-centred. For each component of $G'-S$, identify the new leaf bags that contain $N_G[X]\cup S\cup Z$ to obtain a tree-decomposition of $G$. Since the other bags do not change, it follows that the tree-decomposition is $(20dk,2)$-centred while having a bag that contains $N_G[X]$, as required.
\end{proof}

To conclude, one peculiar artefact of our proof is that we go from balanced separators consisting of balls of radius $1$ to bags in the tree-decomposition consisting of ball of radius $2$. We do not believe that this increase in radius is necessary, and as such, we conjecture that \cref{MainTheorem} can be strengthened so that the tree-decomposition is $(k',1)$-centred. Proving this would be a step towards characterising graphs that have a $(k,1)$-centred tree-decompositions, which would be of independent interest. 

\subsubsection*{Acknowledgement}
This work was completed at the 2025 Oberwolfach Graph Theory Workshop. Thanks to the organisers and participants for providing a stimulating work environment.
 
{
\fontsize{11pt}{12pt}
\selectfont
	
\hypersetup{linkcolor={red!70!black}}
\setlength{\parskip}{2pt plus 0.3ex minus 0.3ex}

\bibliographystyle{DavidNatbibStyle}
\bibliography{main.bbl}
}

\end{document}